 \documentclass[a4paper,13pt,abstracton]{article}
 \usepackage[utf8]{inputenc}
   \providecommand{\keywords}[1]{\textbf{\textit{Key words:}} #1}
   \usepackage{amsmath,amsthm,verbatim,amssymb,amsfonts,amscd, graphicx,mathrsfs}

 \usepackage{geometry}
 \usepackage{tabularx}
 \usepackage{diagbox}
 \numberwithin{equation}{section}
 \newtheorem{thm}{Theorem}[section]
 \newtheorem{lem}[thm]{Lemma}
 \newtheorem{define}[thm]{Definition}
 \newtheorem{cor}[thm]{Corollary}
 
 \newtheorem{rmk}[thm]{Remark}

 \newtheorem{ex}[thm]{Example}
  
  \newtheorem{qu}[thm]{Question}
 \usepackage{enumerate}

 \usepackage[square,comma,super,numbers,sort&compress]{natbib}
 \usepackage{bm}
  \usepackage{hyperref}
 \hypersetup{colorlinks=true, urlcolor=Cerulean, citecolor=red}
 \usepackage[T1]{fontenc}
 \usepackage{nameref,cleveref}
 \usepackage{nicefrac,xfrac}
 \usepackage[all]{xy}
 \usepackage{color}
\usepackage[section]{placeins}
\usepackage{indentfirst}

\begin{document}
\title{\textbf {Morphisms from $\mathbb{P}^m$ to flag varieties}}

\author{Xinyi Fang
\thanks{Department of Mathematics, Shanghai Normal University, Shanghai, 200234, P. R. China, xinyif@shnu.edu.cn}
and Peng Ren \thanks{Shanghai Center for Mathematical Sciences, Fudan University, Jiangwan Campus, Shanghai, 200438, P. R. China, pren@fudan.edu.cn.
The first author is sponsored by National Natural Science Foundation of China (Grant No. 12471040). The second author is supported by NKRD Program of China (No. 2020YFA0713200) and LNMS.
}}

\date{}
\maketitle

%----------------------------------------------------------------------------------------
%	ESSAY BODY
%----------------------------------------------------------------------------------------
%%%%%%%%%%%%%%%%%%%%%%%%%%%%%%%%%%%%%%%%%%%%%%%%%%%%%%%%%%%%%%%%%%%%%%%%%%%%%%%%%%%%%%%%%%%%%%%%%%%%%%%%%%%%%%%%%%%%%%%%%%%%%%%%%%%%%%%%%%%%%%%%%%%%%%%%%%%%%%% ???è??%%%%%%%%%%%%%%%%%%%%%%%%%%%%%%%%%%

\begin{abstract}
	In this paper, we consider the morphisms from projective spaces to flag varieties. We show that the morphisms can only be constant under some special conditions.
 As a consequence, we prove that the splitting types of unsplit uniform $r$-bundles on $\mathbb{P}^m$ can not be $(a_1,\dots,a_1,a_2,\dots,a_{r-k+1})$ for $1\le k\le m-2$ where  $a_1>a_2>\dots>a_{r-k+1}$.
\end{abstract}

%%%%%%%%%%%%%%%%%%%%%%%%%%%%%%%%%%%%%%%%%%%%%%%%%%%%%%%%%%%%%%%%%%%%%%%%%%%%%%%%%%%%%%%%%%%%%%%%%%%%%%%%%%%%%%%%%%%%%%%%%%%%%%%%%%%%%%%%%%%%%%%%%%%%%%%%%%%%%%% ??????è??%%%%%%%%%%%%%%%%%%%%%%%%%%%%%%%%%%
\section{Introduction}
 The morphisms between projective varieties are important objects in algebraic geometry. In particular, when one or both of the varieties have some rigidity, the morphisms between them have strong constraints. For example,  Lazarsfeld proved that if there is a surjective morphism from a projective space $X$ to a projective manifold, then the morphism must be an automorphism \cite{lazarsfeld1984some}. The case when $X$ is the hyperquadric was proved in \cite{cho1994smooth} and \cite{paranjape1989self}. Then, Hwang and Mok used a different approach to deal with the case of \emph{rational homogeneous spaces} with Picard rank 1 \cite{hwang1999holomorphic}. They proved that either the morphism is an automorphism or the target manifold is a projective space. %These results tell us that if one variety has some rigidity, the morphism also exhibits a degree of rigidity.
%When both varieties have some rigidity, the situation becomes more interesting. A well known result is the morphisms from $\mathbb{P}^m$ to $\mathbb{P}^n$ must be constant when $m>n$. A similar assertion also holds for Grassmannians \cite{naldi2022morphisms, tango1974n}. Later, the result is extended to morphisms between rational homogeneous spaces with the concept of \emph{effective good divisibility} \cite{hu2022effective, munoz2022maximal}.

%There are at least two ways to study morphisms between projective varieties.
 The classic method of studying morphisms between projective varieties is to compare the Chow rings of the two varieties. Tango made use of the structure of the cohomology ring of
$\mathbb{P}^m$ to prove the well known result that every morphism from $\mathbb{P}^m$ to $G(l,n+1)$ is constant if $m>n$ \cite{tango1974n}. The same statement with $\mathbb{P}^m$ replaced by an $m$-dimensional projective variety $M$ with $\dim H^{2s}(M,\mathbb{C})=1$ was obtained by Muñoz, Occhetta, and Conde \cite{munoz2012uniform}. Further, they showed that a weaker property of the cohomology ring was required to obtain Tango-type results, which they called \emph{effective good divisibility} \cite{munoz2020splitting}. This concept is a refined version of the \emph{good divisibility} %by considering only effective classes, which was
introduced by Pan in \cite{pan2015triviality}. By calculating the effective good divisibility of Grassmannians, Naldi and Occhetta proved that the morphisms from $G(l,m)$ to $G(k,n)$ must be constant when $m>n$ \cite{naldi2022morphisms}. Subsequently, Muñoz, Occhetta, and Conde computed the effective good divisibility of rational homogeneous spaces of classic type and gave sufficient conditions for morphisms from projective varieties to  rational homogeneous spaces of classic type to be constant \cite{munoz2022maximal}. Independently, Hu, Li and Liu computed the effective good divisibility of rational homogeneous spaces of general Lie type by different methods \cite{hu2022effective}.

%With the concept of effective good divisibility, they demonstrated that there are no nonconstant morphisms from a projective variety to a rational homogeneous space when the effective good divisibility of the projective variety is bigger than that of the rational homogeneous space . Subsequently, Naldi and Occhetta calculated that the effective good divisibility of $G(l,n+1)$ is $n$ .

On the other hand, the existence of nonconstant morphisms has attracted the attention of mathematicians. In 2002, Kumar considered a series of results on the morphisms between rational homogeneous spaces and he conjectured that the existence of nonconstant morphisms is related to the semisimple ranks of rational homogeneous spaces (See \cite{kumar2023nonexistence}, Conjecture 5). Recently, Bakshi and Parameswaran have given a partial affirmative answer to Kumar's conjecture in \cite{bakshi2023morphisms}. Occhetta and Tondelli also considered the nonconstant morphisms between Grassmannians \cite{occhetta2023morphisms}.

%proved that if $\varphi:G(l,n+1)\rightarrow G(k,n+1)$ is a nonconstant map and $l\ne 1,n$, then $\varphi$ is an isomorphism \cite{occhetta2023morphisms}.showed that there is a nonconstant map from $\mathbb{P}^3$ to the \emph{flag variety} $A_n/P$, where $P$ is a minimal parabolic subgroup of $SL(n+1)$.

In this paper, we study the morphisms from projective spaces to flag varieties.
Set $\Delta(A_n):=\{1,2,\ldots,n\}$. Let  $A_n/P(I)$ be a flag variety corresponding to the subset $I\subset \Delta(A_n)$. Suppose
the maximal linear subspaces of $A_n/P(I)$ are $\mathbb{P}^{m-1}$. Then there is a nonconstant map from $\mathbb{P}^{m-1}$ to $A_n/P(I)$. It's then natural to ask whether there exists a nonconstant map from $\mathbb{P}^{m}$ to $A_n/P(I)$. We answer the question in some special cases.

%\begin{thm}\label{Main}
%Fix integers $m$ and $n~(1<m<n+2)$. Let $A_n/P(I_i)$ be a flag variety with $\Delta(A_n)\backslash I_i=\{i,i+1,\dots,i+m-3\}$.
%	\begin{enumerate}	\item If $i=1$ or $n-m+3$, then every morphism from $\mathbb{P}^m$ to $A_n/P(I_i)$ is constant.	\item If $1 < i < n-m+3$, then\begin{enumerate}\item there is a nonconstant map from $\mathbb{P}^m$ to $A_n/P(I_i)$ when $m$ is odd.\item every morphism from $\mathbb{P}^m$ to $A_n/P(I_i)$ is constant when $n=m$ is even.	\end{enumerate}\end{enumerate}\end{thm}
\begin{thm}\label{Main1}
	Fix integers $n$ and $m~(m>1)$. Let $A_n/P(I_i)$ be a flag variety with $\Delta(A_n)\backslash I_i=\{i,i+1,\dots,i+m-3\}$. If one of the following holds
	\begin{enumerate}
		\item[(1)] $i=1$,
		\item[(2)] $i=n-m+3$,
		\item[(3)] $i=2$ and $n=m$ is even,
\end{enumerate}		
	then every morphism from $\mathbb{P}^m$ to $A_n/P(I_i)$ is constant.	
\end{thm}
\begin{thm}\label{Main2}
With the notations as Theorem \ref{Main1}. If $1<i<n-m+3$ and $m$ is odd, then there is a nonconstant map from $\mathbb{P}^m$ to $A_n/P(I_i)$.
\end{thm}

The morphisms between rational homogeneous spaces are closely related to the classification of \emph{uniform vector bundles}, that is, bundles whose \emph{splitting types} are independent of the chosen line. The notion of a uniform vector bundle appears first in a paper by Schwarzenberger\cite{schwarzenberger1961vector}. Since then, there have been a series of related work\cite{van1971uniform,sato1976uniform,elencwajg1980fibres,ellia1982fibres,ballico1983uniform,du2021vector,munoz2012uniform,du2022vector}. Thanks to \cite{elencwajg1980fibres, sato1976uniform} we know that the only examples of unsplit uniform bundles of rank smaller than or equal to $n$ on $\mathbb{P}^n$ are constructed upon $T_{\mathbb{P}^n}$ (by twisting and dualizing). In 1978, Elencwajg \cite{elencwajg1978fibres} show that unsplit uniform vector bundles of rank 3 over $\mathbb{P}^2$ are of the form $T_{\mathbb{P}^2}(a)\oplus \mathcal{O}_{\mathbb{P}^2}(b)$ or $S^2T_{\mathbb{P}^2}(a)~(a,b\in\mathbb{Z})$. Later, Ballico \cite{ballico1983uniform} and Ellia \cite{ellia1982fibres} independently proved that, up to duality, unsplit uniform $(n+1)$-bundles on $\mathbb{P}^n$ are $T_{\mathbb{P}^n}(a)\oplus \mathcal{O}_{\mathbb{P}^n}(b)$ for $n>2$. However, the classification of higher rank uniform bundles on projective spaces is still widely open.
As an application of Theorem \ref{Main1}, we can prove the following theorem.
\begin{thm}
Let $E$ be a uniform $r$-bundle on $\mathbb{P}^m$ of splitting type
\[
(\underbrace{a_1,\ldots,a_1}_\text{k},a_2,\ldots,a_{r-k+1}),~a_1>a_2>\cdots>a_{r-k+1}.
\]
If $1\le k\le m-2$, then $E$ splits as a direct sum of line bundles.
\end{thm}
\paragraph{Notation and convention}

\begin{itemize}
	\item $A_n$: the simple Lie group $SL(n+1)$;
	\item $\Delta(A_n)$: the set $\{1,2,\ldots,n\}$;
	\item $A_n/P(d_1,\dots,d_s)$: the flag variety parameterizing increasing sequences of $d_i$-dimensional subspaces $V_{d_i}$ of $\mathbb{C}^{n+1}$;
	\item $H_{d_i}$: the $i$-th universal bundle of rank $d_i$ on $A_n/P(d_1,\dots,d_s)$;
	\item $G(l,n)$ the Grassmannian of $l$-dimensional subspaces of $\mathbb{C}^{n}$;
	\item $\Sigma_u(X_1,\dots,X_v)$: the complete homogeneous symmetric polynomial of degree $u$ in $v$;
	\item $\sigma_u(X_1,\dots,X_v)$: the elementary homogeneous symmetric polynomial of degree $u$ in $v$;
	\item $A^\bullet(X)$: the Chow ring of the variety $X$;

	%\item $H^\bullet(X)$: the cohomology ring of $X$ with integer coefficients.	
\end{itemize}

\section{Preliminaries}
Throughout this paper, all algebraic varieties and morphisms will be defined over the field $\mathbb{C}$.

\subsection{Flag varieties and Schubert varieties}
We review some well known facts on flag varieties and Schubert varieties and refer to \cite{B-G-G, Brion, eisenbud20163264} for more details. Let $G$ be a semisimple Lie group and $H$ a fixed maximal torus of $G$. Denote their Lie algebras by
$\mathfrak{g}$ and $\mathfrak{h}$ respectively. Let $\Phi$ be its root system and $\Delta=\{\alpha_1,...,\alpha_n\}\subset\Phi$ a set of fixed simple roots. A subgroup $P$ of $G$ for which $G/P$ is projective, called a \emph{parabolic subgroup}, is determined by a set of simple roots of $G$ in the following way: given a subset $I\subset \Delta$, let $\Phi^+_I$ be the subset of $\Phi$ generated by the simple roots in $\Delta\backslash I$. Then the subspace
\begin{align}\label{para}
	\mathfrak{p}(I):=\mathfrak{h}\oplus\sum_{\alpha\in\Phi^+}\mathfrak{g}_{-\alpha}\oplus\sum_{\alpha\in\Phi^+_I}\mathfrak{g}_{\alpha}
\end{align}
is a parabolic subalgebra of $\mathfrak{g}$, determining a parabolic subgroup $P(I)\subset G$. Conversely, every parabolic subgroup is constructed in this way. A projective quotient $G/P(I)$ of a semisimple Lie group is called a \emph{rational homogeneous space}. In particular, $A_n/P(I)$ is called a \emph{flag variety}, where $A_n$ is the simple Lie group $SL(n+1)$. In this paper, we focus on flag varieties and their Schubert varieties.

Set $\Delta(A_n):=\{1,2,\ldots,n\}$. From now on, we will denote by $A_n/P(I)$ a flag variety corresponding to the subset $I\subset \Delta(A_n)$. When $I=\{d_1,\ldots,d_s\}$, $A_n/P(d_1,\ldots,d_s)$ is actually the set of increasing sequences of linear subspaces
\begin{align}
0\subset V_{d_1}\subset\cdots\subset V_{d_s} \subset \mathbb{C}^{n+1}
\end{align}
such that $\dim(V_{d_i})=d_i$ for $i=1,\ldots,s$.
The two extremal cases correspond to the \emph{complete flag varieties} $A_n/P(1,2,\ldots,n)$ and the \emph{Grassmannians} $A_n/P(i)=G(i,n+1)$. We now introduce Schubert varieties in the flag variety $A_n/P(I)$.
%we also write the flag variety $A_n/P_I$ as $F(d_1,\ldots,d_s,n+1)$ or $X(d_1,d_2-d_1,\ldots,n-d_s)$ (see \cite{}).

Let $W=S_{n+1}$ be the
symmetric group of order $(n+1)!$. It's  the Weyl group of $SL(n+1)$. When $I=\{d_1,\ldots,d_s\}$, we denote the subgroup $W_{P(I)}$ of $W$ by
$$W_{P(I)}=S_{d_1}\times S_{d_2-d_1}\times\cdots\times S_{d_s-d_{s-1}}\times  S_{n+1-d_s}.$$
If we consider the coset space $W/W_{P(I)}$, we can find that each coset in $W/W_{P(I)}$ contains a unique permutation $w$ such that
$w(1)<\cdots<w(d_1), w(d_1+1)<\cdots<w(d_2), \dots, w(d_s+1)<\cdots<w(n+1)$.
Equivalently, $w\preceq wv$ for all $v\in W_{P(I)}$, where $\preceq$ is the Bruhat order (cf. \cite{Brion} Section 1.2). This defines the set $W^{P(I)}$ of minimal representatives of $W/W_{P(I)}$. For every $w\in W^{P(I)}$, the Zariski closure
\[X_w=\overline{B w P(I) / P(I)}\]% ~~(\text{resp.} X^w=\overline{B^- w P(I) / P(I)})\]
is called a %(opposite)
\emph{Schubert variety}, where $B$ is a Borel subgroup of $G$. %and $B^-$ is the Borel subgroup opposite to $B$.
It's known from Bruhat-Chevalley that every flag variety admits a canonical decomposition into Schubert varieties \cite{B-G-G}.

\subsection{The Chow ring of the flag varieties}
From the paper \cite{guyot1985caracterisation}, we know that the Chow ring of a flag variety has a simple presentation in terms of the polynomial ring and its ideal. Let's recall some concepts first.

Let $F:=A_n/P(1,2,\ldots,n)$ be a complete flag variety. For every integer $i~(1\le i\le n)$, there is a natural projection $F\rightarrow G(i,n+1)$. %where $A_n/P_i$ is the usual Grassmannian.
On the complete flag variety $F$, there exists a flag of \emph{universal subbundles}
$$0=H_0\subset H_1\subset H_2\cdots\subset H_n\subset \mathcal{O}_{F}^{\oplus{n+1}},$$
where $H_i$ is the pull back of the universal subbundle on the Grassmannian $G(i,n+1)~(1\le i\le n)$. Denote by $X_i=c_1((H_i/H_{i-1})^\vee)$. Take
\[\mathbb{Z}[X_1,\dots,X_{d_1};\dots;X_{d_{s-1}+1},\dots,X_{d_s}]\]
to be the ring of polynomials in $d_s$ variables with integral coefficients symmetrical in $X_{d_{l-1}+1},\dots,X_{d_{l}}$ $(1\le l\le s)$. Let
\begin{align}
	\Sigma_u(X_1,X_2,\dots,X_v)=\sum_{\substack{\alpha_1+\alpha_2+\dots+\alpha_v=u\\\alpha_i\ge 0}}X_1^{\alpha_1}X_2^{\alpha_2}\dots X_v^{\alpha_v}
\end{align}
be a \emph{complete homogeneous symmetric polynomial} of degree $u$ in $v$ variables
and
\begin{align}
	\sigma_u(X_1,X_2,\dots,X_v)=\sum_{1\le i_1<i_2<\cdots<i_u\le v}X_{i_1}X_{i_2}\cdots X_{i_u},%~0\le u\le v
\end{align}
an \emph{elementary symmetric polynomial} of degree $u$ in $v$ variables.
\begin{thm}[See \cite{guyot1985caracterisation} Theorem 3.2]\label{chowring}
	The Chow ring of the flag variety $A_n/P(d_1,\ldots,d_s)$ is  \[\mathbb{Z}[X_1,\dots,X_{d_1};\dots;X_{d_{s-1}+1},\dots,X_{d_s}]/I,\]
	where $I$ is the ideal generated by the set of polynomials $\Sigma_i(X_1,\dots,X_{d_s})$ for $n+2-d_s \leq i\leq n+1$.
\end{thm}

There is an important class of algebraic cycles in the Chow rings of flag varieties, which play an important role in the proof of our main theorems.%More precisely, we have the following lemma.
\begin{lem}\label{cycle}
	The cycles
	\[\sigma_{t}(X_1,X_2,\dots,X_{d_l})=\sum_{1\le i_1<i_2<\cdots<i_t\le d_l}X_{i_1}X_{i_2}\cdots X_{i_t},\text{ for } 1\le l\le s, ~1\leq t\leq d_l, \]in $A_n/P(d_1,\ldots,d_s)$ are Schubert varieties.
\end{lem}
\begin{proof}
First we notice that for every integer $l~(1\le l\le s)$, there is a natural projection
$$\pi_l: A_n/P(d_1,\ldots,d_s)\rightarrow G(d_l,n+1)$$	
such that the $l$-th universal subbundle $H_{d_l}$ on $A_n/P(d_1,\ldots,d_s)$ is the pullback of the unique universal subbundle on $G(d_l,n+1)$. By Whitney's formula and the splitting
principle, we have that the total Chern class of $H_{d_l}^\vee$ on $A_n/P(d_1,\ldots,d_s)$ is
	\[c(H_{d_l}^\vee)=(1+X_1)\cdots(1+X_{d_l})=1+\sigma_1(X_1,\dots,X_{d_l})+\dots+\sigma_{d_l}(X_1,\dots,X_{d_l}).\]
	On the other hand, we already know that
	every Chern class of the universal subbundle on a Grassmannian is a Schubert variety (cf. \cite{eisenbud20163264} Section 5.6). Since the pullback of the Schubert varieties under the map $\pi_l$ are still Schubert varieties, this concludes the proof.
\end{proof}

\section{Morphisms from projective spaces to flag vareities}
Let $A_n/P(I)$ be a flag variety. It's classically known that $A_n/P(I)$ is covered by projective lines and each class of lines can be realized geometrically by considering the double fibration (see \cite{landsberg2003projective} Section 4). In fact, for higher dimensional linear spaces on $A_n/P(I)$, we can still describe their geometries in terms of Tits fibrations. Fix an integer $m~(m\ge 2)$. If %$\mathcal{D}\backslash I=
$\#(\Delta(A_n)\backslash I)=m-2$, then \cite{landsberg2003projective} Corollary 4.10 tells us that the dimension of largest linear spaces on $A_n/P(I)$ is less than or equal to $m-1$. In particular, if $\Delta(A_n)\backslash I$ is a set of consecutive integers, then the largest linear space on $A_n/P(I)$ is exactly a $\mathbb{P}^{m-1}$. In this paper, we focus on the case that  $\Delta(A_n)\backslash I$ is a set of consecutive integers. Since $\mathbb{P}^{m-1}$ is the largest linear space on $A_n/P(I)$, there is a noncontant map from $\mathbb{P}^{m-1}$ to $A_n/P(I)$. Further, we naturally consider the maps between $\mathbb{P}^{m}$ and $A_n/P(I)$. We start with a special case.

\begin{thm}\label{main1}
	%Fix integers $m$ and $n~(1<m<n+2)$. Let $A_n/P(I)$ be a flag variety with $I=\{1,2,\dots,n-m+2\}$. Then every morphism from $\mathbb{P}^m$ to $A_n/P(I)$ is constant.
Fix integers $n$ and $m~(m>1)$. Let $A_n/P(1,2,\dots,n-m+2)$ be a flag variety. Then every morphism from $\mathbb{P}^m$ to $A_n/P(1,2,\dots,n-m+2)$ is constant.	
\end{thm}

\begin{proof}
Let $I=\{1,2,\dots,n-m+2\}$ and $f$ a morphism from $\mathbb{P}^m$ to $A_n/P(I)$. According to Theorem \ref{chowring}, we know that the Chow ring of $A_n/P(I)$ has $n-m+2$ generators $X_1,X_2,\dots,X_{n-m+2}$ %with every $X_i$ is a divisor on $A_n/P(I)$
and
 \[A^\bullet(A_n/P(I))=\mathbb{Z}[X_1;X_2;\dots;X_{n-m+2}]/<\Sigma_i(X_1,\dots,X_{n-m+2})>,  i=m,\dots,n+1.\]
Since  $$A^\bullet(\mathbb{P}^m)=\mathbb{Z}[\mathcal{H}]/<\mathcal{H}^{m+1}>,$$ we can assume $$f^*(X_i)=a_i,~(1\le i\le n-m+2)$$
 for some $a_i\in \mathbb{Z}$. Since $\Sigma_m(X_1,\dots,X_{n-m+2})=0$ in $A^\bullet(A_n/P(I))$, we have

 %Suppose there exists a non-constant map $f$ from $\mathbb{P}^m$ to $X$. Let  Since the map is not constant map, at least one $a_i$ is nonzero. By Lemma \ref{cycle}, we know that $\sigma_u(X_1,\dots,X_k)$ represents Schubert cycles. Therefore, $\sigma_u(a_1,\dots,a_n)=f^*\sigma_u(X_1,\dots,X_n)$ is also an effective algebraic cycle. Thus, we have \[\sigma_u(a_1,\dots,a_n)\geq0.\] Additionally, we can observe that
 \begin{align}\label{se}
 \Sigma_m(a_1,\dots,a_{n-m+2})=f^*(\Sigma_m(X_1,\dots,X_{n-m+2}))=0.
 \end{align}
We separate our discussion into two cases.

Case I: $m$ is even. By Theorem 1.(i) in \cite{Tao2017}, we always have $\Sigma_m(a_1,\dots,a_{n-m+2})\geq 0$, and the equality holds (see equation (\ref{se})) if and only if $a_i=0$, which implies that $f$ must be a constant map.

Case II: $m$ is odd. We first prove the following identity.

\textbf{Claim.}\label{tech} For any integers $j$ and $k~(1\le j\le k)$, we have
	\[\Sigma_j(X_1,\dots,X_k)=\sum_{u+2v=j}\sigma_u(X_1,\dots,X_k)Q_v(X_1,\dots,X_k),\]
where $Q_v(X_1,\dots,X_k)=\Sigma_v(X_1^2,\dots,X_k^2)$.

In fact, the proof of the claim relies on the generating functions of the complete homogeneous polynomials and elementary homogeneous polynomials.
	The generating functions of $\Sigma_j$, $\sigma_u$, and $Q_v$ are the formal power series:
\[
\sum_{j=0}^{\infty}\Sigma_j(X_1,\dots,X_k)t^j=\prod\limits_{i=1}^{k}\frac{1}{(1-X_it)},
\]

\[
\sum_{u=0}^{\infty}\sigma_u(X_1,\dots,X_k)t^j=\prod\limits_{i=1}^{k}(1+X_it),
\]	
	and
\[
	\sum_{v=0}^{\infty}Q_v(X_1,\dots,X_k)t^{2v}=\sum_{v=0}^{\infty}\Sigma_v(X_1^2,\dots,X_k^2)t^{2v}=\prod\limits_{i=1}^{k}\frac{1}{(1-X_i^2t^2)}.
	\]
It is easy to see that the right-hand side of the first equation is equal to the product of the right-hand sides of the second and third equations.
	By analyzing the coefficients on the left-hand sides of the equations above, we can get the claimed identity.

By the above claim and equation (\ref{se}), we have
\begin{align}\label{se2}
0=\Sigma_m(a_1,\dots,a_{n-m+2})=\sum_{u+2v=m}\sigma_u(a_1,\dots,a_{n-m+2})Q_v(a_1,\dots,a_{n-m+2}).
\end{align}
Since
\begin{align}\label{se3}
Q_v(a_1,\dots,a_{n-m+2})=\Sigma_v(a_1^2,\dots,a_{n-m+2}^2)=\sum_{\alpha_1+\dots+\alpha_{n-m+2}=v}a_1^{2\alpha_1}\dots a_{n-m+2}^{2\alpha_{n-m+2}},
\end{align}
we have $Q_v(a_1,\dots,a_{n-m+2})\geq 0$.

On the other hand, by Lemma \ref{cycle}, we know that $\sigma_u(X_1,\dots,X_{n-m+2})$ represents a Schubert cycle. Therefore, $\sigma_u(a_1,\dots,a_{n-m+2})=f^*\sigma_u(X_1,\dots,X_{n-m+2})$ is also an algebraic cycle. Thus, we have \[\sigma_u(a_1,\dots,a_{n-m+2})\geq0.\]
Together with (\ref{se2}) and (\ref{se3}), it is easy to see
\[\sigma_u(a_1,\dots,a_{n-m+2})Q_v(a_1,\dots,a_{n-m+2})=0\]
for any $u$, $v~(u+2v=m)$. In particular,
\[
\sigma_1(a_1,\dots,a_{n-m+2})Q_{\frac{m-1}{2}}(a_1,\dots,a_{n-m+2})=0,
\]
which implies that either \[
Q_{\frac{m-1}{2}}(a_1,\dots,a_{n-m+2})=0~\text{or}~ \sigma_1(a_1,\dots,a_{n-m+2})=0.
\]
If the former condition holds,  %\[H_{\frac{m-1}{2}}(a_1,\dots,a_{n-m+2})=0,\]
 it's obvious that $a_i=0$ for any $1\leq i\leq n-m+2$ from (\ref{se3}).

If the latter condition happens, we have
\[
\Sigma_2(a_1,\dots,a_{n-m+2})+\sigma_2(a_1,\dots,a_{n-m+2})=\sigma_1^2(a_1,\dots,a_{n-m+2})=0.
\]
Since $\sigma_2(a_1,\dots,a_{n-m+2})\ge 0$ and $\Sigma_2(a_1,\dots,a_{n-m+2})\ge 0$ (see \cite{Tao2017} Theorem 1.(i)), we have
\[
\Sigma_2(a_1,\dots,a_{n-m+2})=0
\]
and hence $a_i=0$ for any $1\leq i\leq n-m+2$.
Arguing as above we conclude that $f$ must be a constant map.
\end{proof}
\begin{rmk}
When $m=2$, $3$ or $n+1$, Theorem \ref{main1} leads to some known results.
\begin{enumerate}
	\item[(i)] When $m=2$, Theorem \ref{main1} tells us that every morphism from $\mathbb{P}^2$ to the complete flag variety $A_n/P(1,2,\ldots,n)$ is constant (see  \cite{bakshi2023morphisms} Theorem 1.1).
	\item[(ii)] When $m=3$, every morphism from $\mathbb{P}^3$ to  $A_n/P(1,2,\ldots,n-1)$ is constant (see also \cite{bakshi2023morphisms} Theorem 1.3).
	\item[(iii)] When $m=n+1$, every morphism from $\mathbb{P}^{n+1}$ to $\mathbb{P}^{n}$ is constant. This leads to the well known result that
	 morphisms from $\mathbb{P}^{a}$ to $\mathbb{P}^{b}$ must be constant whenever $a>b$.
\end{enumerate}
\end{rmk}

By the duality of $\mathbb{C}^{n+1}$, there is a natural isomorphism
\[
A_n/P(1,2,\dots,n-m+2)\cong A_n/P(m-1,m,\dots,n).
\]
Hence we have the following corollary.
\begin{cor}\label{cor}
	There is no nonconstant map from $\mathbb{P}^m$ to $A_n/P(m-1,m,\dots,n)$.
\end{cor}

Let $I_i$ be a subset of $\Delta(A_n)$ whose complement is a set of consecutive integers $\{i,i+1,\dots,i+m-3\}$. From Theorem \ref{main1} and Corollary \ref{cor}, we know that there is no nonconstant map from $\mathbb{P}^m$ to $A_n/P(I_i)$ when $i=1$ or $n-m+3$. We are now ready to consider the general cases $1<i<n-m+3$ in terms of the parity of $m$. When $m$ is even, we can handle the extreme case where $n=m$ and $i=2$.
\begin{thm}\label{main2} 	
	If $m$ is even, then every morphism from $\mathbb{P}^m$ to $A_m/P(1,m)$ is constant. 	
\end{thm}
\begin{proof}
	Let $m=2k$. In this case, we have:
	\[A^\bullet(A_{2k}/P(1,2k))=\mathbb{Z}[X_1;X_2,\dots,X_{2k}]/<\Sigma_i(X_1,\dots,X_{2k})>,~i=2,\dots,2k+1.\]
	Suppose $f$ is a morphism from $\mathbb{P}^{2k}$ to $A_{2k}/P(1,2k)$. Let
	$$f^*(X_1)=a~\text{and}~f^*\sigma_j(X_2,\dots,X_{2k})=b_j,~1\le j\le 2k-1.$$
	(By convention we set $b_0=1$ and $b_{2k}=0$.) By the effectivity of Schubert varieties, we have $a\ge 0$ and $a+b_1\ge 0$.
	
	Considering the generating function of the complete homogeneous polynomials, we find
	\[\sum_{i=0}^{\infty}\Sigma_i(X_2,\dots,X_{2k})t^i=(1-X_1t)\sum_{i=0}^{\infty}\Sigma_i(X_1,\dots,X_{2k})t^i.\]
	Thus
	\[\Sigma_i(X_2,\dots,X_{2k})=\Sigma_i(X_1,\dots,X_{2k})-X_1\Sigma_{i-1}(X_1,\dots,X_{2k})~\text{for}~i>1.\]
	Since $\Sigma_i(X_1,\dots,X_{2k})=0~(i>1)$ in $A^\bullet(A_{2k}/P(1,2k))$, we have
	\begin{align}\label{seq1}
		\Sigma_2(X_2,\dots,X_{2k})=-X_1\Sigma_1(X_1,\dots,X_{2k}),
	\end{align}
	and
	\begin{align}\label{seq2}
		\Sigma_i(X_2,\dots,X_{2k})=0 \text{ for } i\geq 3.
	\end{align}
	From equation (\ref{seq1}), we deduce that
	\begin{align}\label{seq3}
		f^*\Sigma_2(X_2,\dots,X_{2k})=-a(a+b_1).
	\end{align}
	On the other hand, the equation $$\sum_{i=0}^{\infty}\Sigma_i(X_2,\dots,X_{2k})t^i\cdot \sum_{i=0}^{\infty}\sigma_i(X_2,\dots,X_{2k})(-t)^i=1$$
	implies that for any $l\ge 1$,
	\[\sum_{i=0}^l(-1)^i\sigma_i(X_2,\dots,X_{2k})\Sigma_{l-i}(X_2,\dots,X_{2k})=0.\]
	So, using (\ref{seq2}) we have
	\begin{small}
		\[\sigma_l(X_2,\dots,X_{2k})=\sigma_{l-1}(X_2,\dots,X_{2k})\Sigma_1(X_2,\dots,X_{2k})-\sigma_{l-2}(X_2,\dots,X_{2k})\Sigma_{2}(X_2,\dots,X_{2k}).\]
	\end{small}
	Applying $f^*$ to the above sequence and combining (\ref{seq3}) we get
	\begin{align}\label{seq4}
		b_l=b_{l-1}b_1+b_{l-2}(a^2+ab_1)~(l\ge 2).
	\end{align}
	By induction on $l$, it is not hard to see that\begin{enumerate}
		\item[1)] if $l$ is odd, then $b_1b_l\geq0$;
		\item[2)] if $l$ is even, then $b_l\geq0.$
	\end{enumerate}
	% for $l=2k$, let $b_{2k}=f^*\sigma_{2k}(X_2,\dots,X_{2k})=0$. Then we have
	Especially, $b_{2k-1}b_1\ge 0$ and $b_{2k-2}(a^2+ab_1)\ge 0$. Since
	\[0=b_{2k}=b_{2k-1}b_1+b_{2k-2}(a^2+ab_1),\]
	we have \[b_{2k-1}b_1=b_{2k-2}(a^2+ab_1)=0.\]
	If $a^2+ab_1\ne 0$, then $b_{2k-2}$ would be equal to $0$. Applying the above  argument repeatedly to (\ref{seq4}), we can get
	\[
	b_{2k-2}=b_{2k-4}=\cdots=b_2=0.
	\]
	Since $b_2=b_1^2+a^2+ab_1$, $a$ and $b_1$
	would be equal to $0$ contrary to hypothesis. Hence $a^2+ab_1=0$. Using (\ref{seq4}), we have $b_l=b_1^l$ for any $l\ge 1$. Since $b_{2k-1}b_1=0$, we necessarily have $b_l=0$ for any $l\ge 1$ and $a=0$. Hence $f$ is a constant map.
\end{proof}

By combining Theorem \ref{main1}, Corollary \ref{cor} and Theorem \ref{main2}, we arrive at the first main theorem (Theorem \ref{Main1}) of our paper.

When $m$ is odd, we will show that there is a nonconstant morphism from $\mathbb{P}^m$ to $A_n/P(I_i)$. To this end, we first prove the following lemma and the proof is similar to the proof of Lemma 5.3 in \cite{bakshi2023morphisms}.
\begin{lem}\label{F134}
	There is a nonconstant morphism from $\mathbb{P}^{2k+1}$ to $A_{2k+1}/P(1,2k+1)$.
\end{lem}
\begin{proof}
	Let $V$ be a vector space of dimension $2k+2$. Since there exists a symplectic form in $V$, every one- dimensional subspace $L$ is orthogonal to a $(2k+1)$-dimensional subspace $L^\perp$ containing $L$. Thus, we have a nonconstant morphism from $\mathbb{P}^{2k+1}$ to $A_{2k+1}/P(1,2k+1)$ by sending $L$ to $\left(L, L^{\perp}, V\right)$.
\end{proof}
Our second main result is as follows.
\begin{thm}\label{main3}
Fix integers $n$ and $m~(m>1)$. Let $A_n/P(I_i)$ be a flag variety with $\Delta(A_n)\backslash I_i=\{i,i+1,\dots,i+m-3\}$. 	
If $m$ is odd and $1<i<n-m+3$, then there is a nonconstant map from $\mathbb{P}^m$ to $A_n/P(I_i)$.
\end{thm}
\begin{proof}
Since $1<i<n-m+3$, $i-1\ge 1$ and $i+m-2\le n$. Let's consider the flag variety $A_n/P(J_i)$ with $\{1,2,\dots,n\}\backslash J_i=\{i-1,i,\dots,i+m-2\}$. It's not hard to see that there is a natural projection  $$\pi:A_n/P(I_i)\rightarrow A_n/P(J_i)$$ and every $\pi$-fiber is isomorphic to $A_m/P(1,m)$. So we can identify $A_m/P(1,m)$ as a subspace of $A_n/P(I_i)$. Since $m$ is odd, according to Lemma \ref{F134}, we can construct a nonconstant map from $\mathbb{P}^m$ to $A_m/P(1,m)$ and hence the composite map $\mathbb{P}^m\rightarrow A_n/P(I_i)$ is nonconstant as well.
\end{proof}
\section{Applications}
A straightforward consequence of Theorem \ref{Main1} is the following result.
\begin{cor}
 Let $X$ be a flag variety as in Theorem \ref{Main1}. Let $Y=A/P_J$ be a flag variety with at most $(m-1)$ consecutive integers in $\Delta(A)\backslash J$. Then every morphism from $Y$ to $X$ is a constant map.
\end{cor}
\begin{proof}
Since there are at most $(m-1)$ consecutive integers in $\Delta(A)\backslash J$, $Y$ is covered by $\mathbb{P}^m$ due to \cite{landsberg2003projective} Theorem 4.9 and Corollary 4.10. By Theorem \ref{Main1}, we know that there is no nonconstant map from $\mathbb{P}^m$ to $X$ and hence no nonconstant map from $Y$ to $X$.
\end{proof}

The morphisms from projective spaces to Grassmannians have a long history and it has a nice application in the classification of uniform vector bundles on projective spaces (see \cite{munoz2012uniform, du2021vector, munoz2012uniform, sato1976uniform} for more details).
 Inspired by this, we will use the first main result of our paper to prove the following theorem. To formulate our theorem, we recall a bit of notation now.

 Let $E$ be a rank $r$ bundle on $\mathbb{P}^m$. According to the theorem of Grothendieck, for every $l\in G(2,m+1)$ there is an $r$-tuple
 \[
 a_E(l)=(a_1(l),\ldots,a_r(l))\in\mathbb{Z}^r;~a_1(l)\ge\cdots\ge a_r(l)
 \]
 with $E|L\cong \oplus_{i=1}^{r}\mathcal{O}_L(a_i(l))$. In this way the mapping
 \[
 a_E:G(2,m+1)\rightarrow \mathbb{Z}^r
 \]
 is defined, $a_E(l)$ is called the splitting type of $E$ on $L$.
 \begin{define}
 $E$ is uniform if $a_E$ is constant.
 \end{define}
\begin{ex}
Let $T_{P^m}$ be the tangent bundle of $\mathbb{P}^m$ and $H$ a hyperplane in $\mathbb{P}^m$. Then $N_{H/\mathbb{P}^n}=\mathcal{O}_H(1)$ and we have the exact sequence
$$0\rightarrow T_H\rightarrow T_{\mathbb{P}^m}|H\rightarrow\mathcal{O}_H(1)\rightarrow 0.$$
Since $Ext^1_{H}(\mathcal{O}_H(1),T_H)=0$, the above exact sequence splits, i.e.
$$ T_{\mathbb{P}^m}|H\cong T_H\oplus \mathcal{O}_H(1).$$	
It's known that $T_L=\mathcal{O}_L(2)$ for every line $L\subset\mathbb{P}^m$, so it follows by induction over $m$ that the tangent bundle $T_{\mathbb{P}^m}$ is uniform of splitting type $(2,1,\ldots,1)$.	
\end{ex}
Inspired by the above example, a natural question arises:
\begin{qu}
    Is there an unsplit uniform bundle on $\mathbb{P}^m$ whose splitting type containing k identical numbers $(1\leq k\leq m-2)?$
\end{qu}
 The following theorem partially answers this question.
\begin{thm}\label{main4}
%Let $E$ be a uniform $r$-bundle on $\mathbb{P}^m$. If for any line $L$ in $\mathbb{P}^m$,
%\[
%E|L=\mathcal{O}_L(a_1)^k\oplus \mathcal{O}_L(a_2)\oplus\mathcal{O}_L(a_3)\oplus\cdots\oplus\mathcal{O}_L(a_{r-k+1}),~a_1>a_2>\cdots>a_{r-k+1},
%\]
%and $1\le k\le m-2$, then $E$ splits as a direct sum of line bundles.

Let $E$ be a uniform $r$-bundle on $\mathbb{P}^m$ of splitting type
\[
(\underbrace{a_1,\ldots,a_1}_\text{k},a_2,\ldots,a_{r-k+1}),~a_1>a_2>\cdots>a_{r-k+1}.
\]
If $1\le k\le m-2$, then $E$ splits as a direct sum of line bundles.
\end{thm}
\begin{proof}
We prove this theorem by induction on $r$. For $1\le r\le m-1$, since every uniform $r$-bundle split (see \cite{van1971uniform, sato1976uniform} for more details), the assertion is true. Suppose the assertion is true for all uniform $(r-1)$-bundles of splitting type as above. Let's consider the following standard diagram:
\begin{align}\label{key}
	\xymatrix{
		A_m/P(1,2)\ar[d]^{p}   \ar[r]^-{q} & G(2,m+1)\\
		\mathbb{P}^m.
	}
\end{align}
For $l\in A_m/P(1,2)$, the $q$-fibre
\[\widetilde{L}={q}^{-1}(l)=\{(x,l)|x\in L\}\]
is mapped isomorphically under $p$ to the line $L$ in $\mathbb{P}^m$ determined by $l$ and we have \[p^\ast E|\widetilde{L}\cong E|L.\]
For $x\in G$, the $p$-fibre over $x$,
\[p^{-1}(x)=\{(x,l)|x\in L\}
\]
is mapped isomorphically under $q$ to $\mathbb{P}^{m-1}$. Since for any line $L$ in $\mathbb{P}^m$,
\[
E|L=\mathcal{O}_L(a_1)^k\oplus \mathcal{O}_L(a_2)\oplus\cdots\oplus\mathcal{O}_L(a_{r-k+1}),~a_1>a_2>\cdots>a_{r-k+1},
\]
there is a filtration
\[
0\subset F_1\subset\cdots\subset F_{r-k}\subset F_{r-k+1}=p^\ast E
\]
of $p^\ast E$ by subbundles $F_i$, whose restriction to $q$-fibres $\widetilde{L}$ look as follows:
\begin{align*}
F_i|\widetilde{L}\cong \mathcal{O}_{\widetilde{L}}(a_1)^k\oplus \mathcal{O}_{\widetilde{L}}(a_2)\oplus\cdots\oplus\mathcal{O}_L(a_{i}).
\end{align*}
This filtration is the relative Harder-Narasimhan filtration of $p^\ast E$ (cf. \cite{elencwajg1980fibres} page 38). Because $F_i$ are rank $(k+i-1)$ subbundles of $p^\ast E$, for every point $x\in G$, they provide a morphism
\begin{align*}
\varphi:\mathbb{P}^{m-1}=p^{-1}(x)&\rightarrow A_{r-1}/P(k,k+1,\ldots,r-1)\\
(x,l)&\mapsto F_1(x,l)\subset F_2(x,l)\subset\ldots\subset F_{r-k}(x,l)\subset E(x),
\end{align*}
where $F_i(x,l)$ are the fibres of $F_i$ over $(x,l)$. Since $1\le k\le m-2$,
$\varphi$ must be constant by Theorem \ref{Main1}. Thus $F_i$ are trivial on all $p$-fibres. On the other hand, let $F'=p^\ast E/ F_{r-k}$, then on $A_m/P(1,2)$ we have the following exact sequence
\begin{align}\label{ex}
0\rightarrow F_{r-k}\rightarrow p^\ast E\rightarrow F'\rightarrow0.
\end{align}
Because for every point $x\in G$, $F'|p^{-1}(x)$ is globally generated and $$c_1(F'|p^{-1}(x))=c_1(p^\ast E|p^{-1}(x))-c_1(F_{r-k}|p^{-1}(x))=0,$$ $F'|p^{-1}(x)$ is trivial due to \cite{du2022vector} Corollay 3.2. Since $F_{r-k}$ and $F'$ are trivial on all $p$-fibres, by the base-change theorem $E_{r-k}:=p_{\ast}F_{r-k}$ is a rank $(r-1)$ vector bundle and  $E':=p_{\ast}F'$ is a line bundle on $\mathbb{P}^m$. Moreover
 the canonical morphisms $p^{\ast}p_{\ast}F_{r-k}\rightarrow F_{r-k}$, $p^{\ast}p_{\ast}F'\rightarrow F'$ are then isomorphisms. By projecting the bundle sequence (\ref{ex}) onto $\mathbb{P}^m$ we get the exact sequence
 \begin{align}\label{ex2}
 0\rightarrow E_{r-k}\rightarrow  E\rightarrow E'\rightarrow0.
 \end{align}
 Because $E_{r-k}$ is a uniform $(r-1)$ bundle of splitting type $(a_1,\ldots,a_1,a_2,\ldots,a_{r-k})$,
  by the induction hypothesis, $E_{r-k}$ is a direct sum of line bundles. And because $E'$ is a line bundle, $H^1(\mathbb{P}^m,E'^{\vee}\otimes E_{r-k})=0$. Thus the exact sequence (\ref{ex2}) splits and hence also $E$.
\end{proof}

As a consequence of Theorem \ref{main4}, we obtain the following corollary by dualizing.
\begin{cor}
Let $E$ be a uniform $r$-bundle on $\mathbb{P}^m$ of splitting type $$(b_1,\ldots,b_{r-k},b_{r-k+1}, \ldots, b_{r-k+1}),~b_1>\cdots>b_{r-k}>b_{r-k+1}.$$
If $1\le k\le m-2$, then $E$ splits as a direct sum of line bundles.
\end{cor}

The proof of Theorem \ref{main4} suggests that we may be able to judge whether the maps from projective spaces to flag varieties are constant in terms of the possible splitting type of uniform bundles. Theorem \ref{main2} in our paper and \cite{bakshi2023morphisms} Theorem 1.5 allow us to ask the following question.
\begin{qu}
With the notations as Theorem \ref{Main1}. Let $1<i<n-m+3$.
\begin{enumerate}
	\item[(a)] If $m$ is even, is any morphism from $\mathbb{P}^m$ to $A_n/P(I_i)$ constant?
	\item[(b)] If $m$ is odd, is any morphism from $\mathbb{P}^{m+1}$ to $A_n/P(I_i)$ constant?
\end{enumerate}
\end{qu}

\section*{Acknowledgements}
All the authors thank Professor Rong Du and Changzheng Li for giving some useful suggestions.

\bibliographystyle{abbrv}

\bibliography{ref}

\begin{thebibliography}{10}

\bibitem{bakshi2023morphisms}
S.~Bakshi and A.~J. Parameswaran.
\newblock Morphisms from projective spaces to \text{$G/P$}.
\newblock {\em arXiv preprint arXiv:2308.00286}, 2023.

\bibitem{ballico1983uniform}
E.~Ballico.
\newblock Uniform vector bundles of rank {$(n+1)$}\ on $\mathbb{P}^n$.
\newblock {\em Tsukuba J. Math.}, 7(2):215--226, 1983.

\bibitem{B-G-G}
I.~N. Bernstein, I.~M. Gelfand, and S.~I. Gelfand.
\newblock Schubert cells, and the cohomology of the spaces {$G/P$}.
\newblock {\em Russ. Math. Surv.}, 28(3):3--26, 1973.

\bibitem{Brion}
M.~Brion.
\newblock Lectures on the geometry of flag varieties.
\newblock In {\em Topics in cohomological studies of algebraic varieties},
  Trends Math., pages 33--85. Birkh\"auser, Basel, 2005.

\bibitem{cho1994smooth}
K.~Cho and E.-i. Sato.
\newblock Smooth projective varieties dominated by smooth quadric hypersurfaces
  in any characteristic.
\newblock {\em Math. Z.}, 217(4):553--565, 1994.

\bibitem{du2021vector}
R.~Du, X.~Fang, and Y.~Gao.
\newblock Vector bundles on rational homogeneous spaces.
\newblock {\em Ann. Mat. Pura Appl.}, 200(6):2797--2827, 2021.

\bibitem{du2022vector}
R.~Du, X.~Fang, and Y.~Gao.
\newblock Vector bundles on flag varieties.
\newblock {\em Math. Nachr.}, 296(2):630--649, 2023.

\bibitem{eisenbud20163264}
D.~Eisenbud and J.~Harris.
\newblock {\em 3264 and all that---a second course in algebraic geometry}.
\newblock Cambridge University Press, Cambridge, 2016.

\bibitem{elencwajg1978fibres}
G.~Elencwajg.
\newblock Les fibr\'es uniformes de rang {$3$} sur $\mathbb{P}^2(\mathbb{C})$
  sont homog\`enes.
\newblock {\em Math. Ann.}, 231(3):217--227, 1978.

\bibitem{elencwajg1980fibres}
G.~Elencwajg, A.~Hirschowitz, and M.~Schneider.
\newblock Les fibr{\'e}s uniformes de rang au plus n sur
  $\mathbb{P}^n(\mathbb{C})$ sont ceux qu’on croit.
\newblock {\em Progr. Math.}, 7:37--63, 1980.

\bibitem{ellia1982fibres}
P.~Ellia.
\newblock Sur les fibr\'es uniformes de rang {$(n+1)$}\ sur $\mathbb{P}^n$.
\newblock {\em M\'em. Soc. Math. France}, 7:1--60, 1982.

\bibitem{guyot1985caracterisation}
M.~Guyot.
\newblock Caract\'erisation par l'uniformit\'e{} des fibr\'es universels sur la
  grassmanienne.
\newblock {\em Math. Ann.}, 270(1):47--62, 1985.

\bibitem{hu2022effective}
H.~Hu, C.~Li, and Z.~Liu.
\newblock Effective good divisibility of rational homogeneous varieties.
\newblock {\em Math. Z.}, 305(52), 2023.

\bibitem{hwang1999holomorphic}
J.-M. Hwang and N.~Mok.
\newblock Holomorphic maps from rational homogeneous spaces of {P}icard number
  {$1$} onto projective manifolds.
\newblock {\em Invent. Math.}, 136(1):209--231, 1999.

\bibitem{kumar2023nonexistence}
S.~Kumar.
\newblock Nonexistence of regular maps between homogeneous projective
  varieties.
\newblock {\em arXiv preprint arXiv:2307.07018}, 2023.

\bibitem{landsberg2003projective}
J.~M. Landsberg and L.~Manivel.
\newblock On the projective geometry of rational homogeneous varieties.
\newblock {\em Comment. Math. Helv.}, 78(1):65--100, 2003.

\bibitem{lazarsfeld1984some}
R.~Lazarsfeld.
\newblock Some applications of the theory of positive vector bundles.
\newblock In {\em Complete intersections}, volume 1092 of {\em Lecture Notes in
  Math.}, pages 29--61. Springer, Berlin, 1984.

\bibitem{munoz2012uniform}
R.~Mu\~noz, G.~Occhetta, and L.~E. Sol\'a{}~Conde.
\newblock Uniform vector bundles on {F}ano manifolds and applications.
\newblock {\em J. Reine Angew. Math.}, 664:141--162, 2012.

\bibitem{munoz2020splitting}
R.~Mu\~noz, G.~Occhetta, and L.~E. Sol\'a{}~Conde.
\newblock Splitting conjectures for uniform flag bundles.
\newblock {\em Eur. J. Math.}, 6(2):430--452, 2020.

\bibitem{munoz2022maximal}
R.~Mu\~noz, G.~Occhetta, and L.~E. Sol\'a{}~Conde.
\newblock Maximal disjoint {S}chubert cycles in rational homogeneous varieties.
\newblock {\em Math. Nachr.}, 297(1):174--194, 2024.

\bibitem{naldi2022morphisms}
A.~Naldi and G.~Occhetta.
\newblock Morphisms between {G}rassmannians.
\newblock {\em Proc. Japan Acad. Ser. A Math. Sci.}, 98(10):101--105, 2022.

\bibitem{occhetta2023morphisms}
G.~Occhetta and E.~Tondelli.
\newblock Morphisms between {G}rassmannians, {II}.
\newblock {\em Arch. Math.}, 122(5):521--529, 2024.

\bibitem{pan2015triviality}
X.~Pan.
\newblock Triviality and split of vector bundles on rationally connected
  varieties.
\newblock {\em Math. Res. Lett.}, 22(2):529--547, 2015.

\bibitem{paranjape1989self}
K.~H. Paranjape and V.~Srinivas.
\newblock Self-maps of homogeneous spaces.
\newblock {\em Invent. Math.}, 98(2):425--444, 1989.

\bibitem{sato1976uniform}
E.-i. Sato.
\newblock Uniform vector bundles on a projective space.
\newblock {\em J. Math. Soc. Japan}, 28(1):123--132, 1976.

\bibitem{schwarzenberger1961vector}
R.~L.~E. Schwarzenberger.
\newblock Vector bundles on the projective plane.
\newblock {\em Proc. London Math. Soc.}, 11(3):623--640, 1961.

\bibitem{tango1974n}
H.~Tango.
\newblock On {$(n-1)$}-dimensional projective spaces contained in the
  {G}rassmann variety {${\rm Gr}(n,\,1)$}.
\newblock {\em J. Math. Kyoto Univ.}, 14:415--460, 1974.

\bibitem{Tao2017}
T.~Tao.
\newblock Schur convexity and positive definiteness of the even degree complete
  homogeneous symmetric polynomials, 2017.
\newblock
  https://terrytao.wordpress.com/2017/08/06/schur-convexity-and-positive-definiteness-of-the-even-degree-complete-homogeneous-symmetric-polynomials.

\bibitem{van1971uniform}
A.~Van~de Ven.
\newblock On uniform vector bundles.
\newblock {\em Math. Ann.}, 195:245--248, 1972.

\end{thebibliography}
\end{document}